\documentclass[11pt,a4paper]{amsart}
\usepackage{mathrsfs}
\usepackage{syntonly}
\usepackage{amsmath}
\usepackage{amsthm}
\usepackage{amsfonts}
\usepackage{amssymb}
\usepackage{latexsym}

\usepackage{amscd,amssymb,amsopn,amsmath,amsthm,graphics,amsfonts,mathrsfs,accents,enumerate,verbatim,calc}
\usepackage[dvips]{graphicx}
\usepackage[colorlinks=true,linkcolor=red,citecolor=blue]{hyperref}
\usepackage[all]{xy}
\usepackage{enumitem}

\date{}
\allowdisplaybreaks[4] \footskip=15pt
\renewcommand{\uppercasenonmath}[1]{}
\renewcommand\baselinestretch{1.2}

\renewcommand\thesection{\arabic{section}}
\numberwithin{equation}{section} \theoremstyle{plain}
\newtheorem{theorem}{Theorem}[section]
\newtheorem{corollary}[theorem]{Corollary}
\newtheorem{lemma}[theorem]{Lemma}
\newtheorem{proposition}[theorem]{Proposition}
\theoremstyle{definition}
\newtheorem{definition}[theorem]{Definition}

\newtheorem{setup}[theorem]{Setup}

\newtheorem*{ack*}{ACKNOWLEDGEMENTS}

\newcommand{\oo}{\otimes}
\newcommand{\pf}{\noindent\begin {proof}}
\newcommand{\epf}{\end{proof}}

\newcommand{\ccc}[3]{{
  \left(\begin{smallmatrix} {#1}   \\   {#2} \\\end{smallmatrix}\right)_{#3}}}

\newcommand{\Hom}{{\rm Hom}}

\begin{document}
\begin{center}
{\Large  \bf Gorenstein modules respect to duality pairs over triangular matrix rings}

\vspace{0.5cm}   Haiyu Liu and Rongmin Zhu\footnote{Corresponding author.
} \\
\medskip

\end{center}

\bigskip
\centerline { \bf  Abstract}
\medskip

\leftskip10truemm \rightskip10truemm \noindent
Let $A$, $B$ be two rings and $T=\left(\begin{smallmatrix}  A & M \\  0 & B \\\end{smallmatrix}\right)$ with $M$ an $A$-$B$-bimodule. We first construct a semi-complete duality pair $\mathcal{D}_{T}$ of $T$-modules using duality pairs in $A$-Mod and $B$-Mod respectively. Then we characterize when a left $T$-module is Gorenstein $D_{T}$-projective, Gorenstein $D_{T}$-injective or Gorenstein $D_{T}$-flat. These three class of $T$-modules will induce model structures on $T$-Mod. Finally we show that
the homotopy category of each of model structures above admits a recollement relative to corresponding stable categories. Our results give new characterizations to earlier results in this direction. \\
\vbox to 0.3cm{}\\

\leftskip10truemm \rightskip10truemm \noindent
\leftskip10truemm \rightskip10truemm \noindent
\hspace{1em} \\[2mm]

\leftskip0truemm \rightskip0truemm
\vspace{-5em}
\section { \bf Introduction}
 Let $A$ and $B$ be two rings. For any bimodule $_{A}M_{B}$, we write $T$ for
the upper triangular matrix ring $\left(\begin{smallmatrix}  A & M \\  0 & B \\\end{smallmatrix}\right)$. Such rings play an important role in the
study of the representation theory of artin rings and algebras.
Some important
classes of modules over upper triangular matrix rings have been studied by many authors (e.g., see \cite{ah00}, \cite{ah99}, \cite{xi12}, \cite{zh13} and \cite{ee14} and their references).
For example, Zhang \cite{zh13} explicitly described the Gorenstein projective modules over
a triangular matrix Artin algebra. Enochs and other authors  \cite{ee14} characterized when
a left module over a triangular matrix ring is Gorenstein projective or Gorenstein
injective under the “Gorenstein regular” condition.
Zhu, Liu and Wang \cite[Theorem 3.8]{zhu16} characterized Gorenstein flat modules over a triangular matrix ring $T$. Very recently, Mao \cite[Theorem 2.3]{mao20} further studied Gorenstein flat modules over triangular matrix rings which improves \cite[Theorem 3.8]{zhu16} .

Duality pairs were introduced by Holm-J{\o}rgensen in \cite{HJ09}. Recall that for a given $R$-module $M$, its character module is defined to be the $R$-module $M^{+}=\Hom_{\mathbb{Z}}(M, \mathbb{Q/Z})$. A duality pair  is essentially a pair of classes $\mathcal{(L, A)}$ such that $L\in \mathcal{L}$ if and only if $L^{+}\in\mathcal{A}$, and $\mathcal{A}$ is closed under direct summands and finite direct sums. Recently, Gillespie \cite{gj18} showed that the entire theory of Gorenstein homological algebra, complete with associated abelian model structures with stable homotopy categories, can be done with respect to a complete duality pair.
Assume that $\mathcal{D}_{R}=\mathcal{(L, A)}$ is a complete duality pair. We say that an $R$-module $N$ is
\emph{$\mathcal{D}_{R}$-Gorenstein projective} if $N = Z^{0}\mathbf{P}$ for some exact $\mathrm{Hom}_{R}(-,\mathcal{L})$-acyclic complex of projective $R$-modules $\mathbf{P}$. That is, both $\mathbf{P}$ and $\mathrm{Hom}_{R}(\mathbf{P},L)$ are exact (acyclic)
complexes for all $L\in \mathcal{L}$. Those familiar with Gorenstein homological algebra will
guess the definitions of the other concepts, see Definitions \ref{def2.3} and \ref{def2.4} for precise definitions.  When $R$ is a commutative Noetherian ring of finite Krull dimension, then these definitions, applied to the flat-injective duality pair, agree with the usual definitions of Gorenstein injective, projective and flat modules studied by Enochs and many other authors \cite{EJ1,ee00}. In fact, the requirement for the duality pair to be complete is too strong, so they defined semi-complete duality pairs in  \cite{gj21} and showed that if we applied to the duality pair $\mathcal{D}_{R}=(\langle _{R}\mathcal{F}\rangle, \langle \mathcal{I}_{R}\rangle)$, then these definitions agree with the definitions of projectively coresolved Gorenstein flat \cite{js18}, Ding injective \cite{dm08,gi10} and  Gorenstein flat modules \cite{ee00}.

The main goal of this paper is to study Gorenstein homological modules respect to  semi-complete duality pairs over triangular matrix rings.
Mao \cite{ma20} constructed a  complete duality pair $\mathcal{D}_{T}$ of $T$-modules using duality pairs in $A$-Mod and $B$-Mod respectively. Based on this result, in Section \ref{sec3}, we further study when  $\mathcal{D}_{T}$ is semi-complete. Then we characterize when a left $T$-module is Gorenstein $D_{T}$-projective, Gorenstein $D_{T}$-injective or Gorenstein $D_{T}$-flat (see Theorems \ref{the3.4}, \ref{the3.7} and \ref{the3.8}).
As applications, we investigate when a left $T$-module is projectively coresolved Gorenstein flat, Ding injective or Gorenstein flat. In fact, Mao \cite{mao22} characterized Ding injective modules over a triangular matrix ring $T$ under the condition ``$T$ is  right coherent, $_{A}M$ has finite
flat dimension, $M_{B}$ is finitely presented and has finite projective or FP-injective dimension''. Our result presents a new characterization of Ding injective modules (see Corollary \ref{cor3.8}). One can compare it with \cite[Theorem 4.4]{mao22}.

If we are given two cofibrantly generated model structures $\mathcal{M}_{A}$ and $\mathcal{M}_{B}$ on $A$-Mod and $B$-Mod respectively, we investigated in \cite{zhu20} when there exists a cofibrantly generated model structure $\mathcal{M}_{T}$ on $T$-Mod and a recollement of $\mathrm{Ho}(\mathcal{M}_{T})$ relative to $\mathrm{Ho}(\mathcal{M}_{A})$ and $\mathrm{Ho}(\mathcal{M}_{B})$. Let $\mathcal{D}_{R} = \mathcal{(L,A)}$ be a semi-complete duality pair. By \cite[Corollary 5.1]{gj21}, there are three abelian module structures induced by  $\mathcal{D}_{R}$: the Gorenstein $\mathcal{D}_{R}$-projective, $\mathcal{D}_{R}$-injective and $\mathcal{D}_{R}$-flat model structures.
We know that each of these model structures will gives rise to a stable category of modules. In Section \ref{sec4}, using the characterizations in Section 3, we show that
the homotopy category of each of  model structures above on $T$-Mod admits a recollement relative to corresponding homotopy categories (see Theorem \ref{the4.3}).  Finally, we give some applications of our  results for   projectively coresolved Gorenstein flat, Ding injective and Gorenstein flat model structures.
It should be noticed that the recollements of stable categories of Ding injective modules and Gorenstein flat modules over a triangular matrix ring $T$ have been established in \cite[Theorem 2.10]{ww21} and \cite[Theorem 4.12]{zhu20} respectively. Our result Theorem \ref{the4.3} gives a new criterion for the existence of these two recollements.

\bigskip
\section { \bf Preliminaries }

Throughout this paper, all rings are nonzero associative rings with identity and all modules are unitary. For a ring $R$, we write $R$-Mod (resp. Mod-$R$) for the category of left (resp. right) $R$-modules. $M_{R}$ (resp. $_{R}M$) denotes a right (resp. left) $R$-module.


\subsection{Duality pairs.} \cite[Definition 2.1]{HJ09} A \emph{ duality pair} over a ring $R$ is a pair ($\mathcal{L}, \mathcal{A}$), of classes of $R$-modules, satisfying

(1) $L\in \mathcal{L}$ if and only if $L^{+}\in \mathcal{A}$, and

(2) $\mathcal{A}$ is closed under direct summands and finite direct sums.

If ($\mathcal{L}, \mathcal{A}$) is a duality pair, then $\mathcal{L}$ is closed under pure submodules, pure quotients, and pure extensions.

\begin{definition}\cite[Appendix A]{br14} By a \emph{symmetric duality pair} over $R$ we
mean a pair of classes $\mathcal{(L,A)}$ for which both $\mathcal{(L,A)}$ and $\mathcal{(A,L)}$ are duality pairs. A duality pair ($\mathcal{L}, \mathcal{A}$) is called \emph{perfect} if $\mathcal{L}$ contains the module $R$, and is closed under direct sums and extensions.
\end{definition}

As in \cite{gj21}, we call $\mathcal{(L,A)}$ a \emph{semi-perfect }duality pair if it
has all the properties required to be a perfect duality pair except that $\mathcal{L}$ may not
be closed under extensions.

\begin{definition} \cite[Definition 2.5]{gj21} By a \emph{semi-complete} duality pair $\mathcal{(L,A)}$ we mean that $\mathcal{(L,A)}$ is a
symmetric duality pair with $\mathcal{(L,A)}$ being a semi-perfect duality pair. If $\mathcal{(L,A)}$ is indeed perfect, then we call it a \emph{complete} duality pair.
\end{definition}
 Several examples of perfect and symmetric duality pairs are given in  \cite{de18, gj18, HJ09}.

 \bigskip
\subsection{Gorenstein modules relative to a duality pair.}

Throughout this subsection we let $\mathcal{D}_{R}=(\mathcal{L},\mathcal{A})$ denote a fixed semi-complete duality pair over $R$.

\begin{definition} \label{def2.3} Given an $R$-module $N$, a chain complex $\mathbf{I}$ of injective $R$-modules
is called $N$-\emph{acyclic} if $\Hom_{R}(N,\mathbf{I})$ is exact. In a similar way, given a class $\mathcal{N}$
of $R$-modules, $\mathbf{I}$ will be called $\mathcal{N}$-\emph{acyclic} if it is $N$-acyclic for all $N\in \mathcal{N}$. On the
other hand, if $\mathbf{P}$ is a chain complex of projective (or even flat) $R$-modules, we call
it $N^{\otimes}$-\emph{acyclic} if $N\otimes_{R}\mathbf{P}$ is exact; and similarly we define $\mathcal{N}^{\otimes}$-\emph{acyclic} complexes of projective (or even flat) $R$-modules for a class $\mathcal{N}$.
\end{definition}

\begin{definition} \label{def2.4} An $R$-module $N$ is called

$(1)$ \emph{Gorenstein $(\mathcal{L},\mathcal{A})$-projective} (or Gorenstein $\mathcal{D}_{R}$-projective) if
$N =Z^{0}\mathbf{P}$ for some exact $\mathcal{A}^{\otimes}$-acyclic complex of projectives $\mathbf{P}$. Let $\mathcal{GP}_{\mathcal{D}_{R}}$ denote the class of all Gorenstein $\mathcal{D}_{R}$-projective modules.

$(2)$  \emph{$(\mathcal{L},\mathcal{A})$-Gorenstein projective} (or $\mathcal{D}_{R}$-Gorenstein projective) if
$N=Z^{0}\mathbf{P}$ for some exact complex of projectives $\mathbf{P}$ which remains exact after applying
$\Hom_{R}(-,L)$ for any $L \in \mathcal{L}$. Let $\mathcal{GP}^{\mathcal{D}_{R}}$ denote the class of all $\mathcal{D}_{R}$-Gorenstein projective modules.

$(3)$
An $R$-module $N$ is called \emph{Gorenstein $(\mathcal{L},\mathcal{A})$-injective} (or Gorenstein $\mathcal{D}_{R}$-injective) if
$N =Z^{0}\mathbf{I}$ for some exact $\mathcal{A}$-acyclic complex of injectives $\mathbf{I}$.
Let $\mathcal{GI}_{\mathcal{D}_{R}}$ denote the class
of all Gorenstein $\mathcal{D}_{R}$-injective modules.

$(4)$  An $R$-module $N$ is called \emph{Gorenstein $(\mathcal{L},\mathcal{A})$-flat} (or Gorenstein $\mathcal{D}_{R}$-flat) if
$N = Z^{0}\mathbf{F}$ for some exact $\mathcal{A}^{\otimes}$-acyclic complex of flat modules $\mathbf{F}$. Let $\mathcal{GF}_{\mathcal{D}_{R}}$ denote the class
of all Gorenstein $\mathcal{D}_{R}$-flat modules.
\end{definition}
Note that  a Gorenstein $\mathcal{D}_{R}$-projective module $N$ is always
 Gorenstein $\mathcal{D}_{R}$-flat. If $\mathcal{D}_{R}$ is a symmetric duality pair, then
$\mathcal{GP}_{\mathcal{D}_{R}}=\mathcal{GP}^{\mathcal{D}_{R}}$ by \cite[Theorem A6]{br14}.

 \bigskip
\subsection{Triangular matrix rings.}
 Let $A$, $B$ be two rings and $T=\left(\begin{smallmatrix}  A & M \\  0 & B \\\end{smallmatrix}\right)$ with $M$ an $A$-$B$-bimodule. Next, we recall the description of left $T$-modules via column vectors. Let $X_{1}\in A$-Mod and $X_{2}\in B$-Mod, and let $\phi^{X}:M\otimes_{B} X_{2}\rightarrow X_{1}$ be a homomorphism of left $A$-modules. The left $T$-module structure on $X=\binom{X_{1}}{X_{2}}$ is defined by the following identity
$$\left(\begin{matrix}  a & m\\
0&b\\\end{matrix}\right)\left(\begin{matrix}  x_{1}\\
x_{2}\\\end{matrix}\right)=\left(\begin{matrix}  ax_{1}+\phi^{X}(m\otimes x_{2}) \\
bx_{2}\\\end{matrix}\right),$$
where $a\in A,~b\in B,~m\in M,~x_{i}\in X_{i}$ for $i=1,~2$.	
According to \cite[Theorem 1.5]{gr82}, $T$-Mod is equivalent to the category whose objects are triples $X=\binom{X_{1}}{X_{2}}_{\phi^{X}}$, where $X_{1}\in A$-$\mathrm{Mod}$, $X_{2}\in B$-$\mathrm{Mod}$ and $\phi^{X}:M\otimes_{B} X_{2}\rightarrow X_{1}$ is an $A$-homomorphism, and whose morphisms between two objects $X=\binom{X_{1}}{X_{2}}_{\phi^{X}}$ and $Y=\binom{Y_{1}}{Y_{2}}_{\phi^{Y}}$~are pairs $\binom{f_{1}}{f_{2}}$ such that $f_{1}\in \mathrm{Hom}_{A}(X_{1},Y_{1})$, $f_{2}\in \mathrm{Hom}_{B}(X_{2},Y_{2})$, satisfying that the diagram
$$\xymatrix{
  M\otimes_{B} X_{2} \ar[d]_{\phi^{X}} \ar[r]^{1_{M}\otimes_{B} f_{2}} & M\otimes_{B} Y_{2} \ar[d]_{ \phi^{Y}} \\
  X_{1} \ar[r]^{f_{1}} & Y_{1}   }
$$
is commutative. In the rest of the paper we identify $T$-Mod with this category and, whenever there is no possible confusion, we omit the homomorphism $\phi$. Consequently, throughout the paper, a left $T$-module is a pair $\binom{X_{1}}{X_{2}}$. Given such a module $X$, we denote by $\widetilde{\phi}^{X}$ the morphism from $X_{2}$ to $\mathrm{Hom}_{A}(M,X_{1})$ given by $\widetilde{\phi}^{X}(x)(m)=\phi^{X}( m\otimes x)$ for each $x\in X_{2},~m\in M$.

Note that a sequence of $T$-modules
$$0\rightarrow\left(\begin{smallmatrix} M_{1}' \\   M_{2}' \\\end{smallmatrix}\right)\rightarrow\left(\begin{smallmatrix}  M_{1}  \\   M_{2} \\\end{smallmatrix}\right)\rightarrow\left(\begin{smallmatrix}  M_{1}''  \\   M_{2}'' \\\end{smallmatrix}\right)\rightarrow0$$
is exact if and only if both sequences $0\rightarrow M_{1}'\rightarrow M_{1}\rightarrow M_{1}''\rightarrow0$ of $A$-modules
and $0\rightarrow M_{2}'\rightarrow M_{2}\rightarrow M_{2}''\rightarrow0$ of $B$-modules are exact.

Recall that each right $T$-module is identified with a triple $(X, Y)_{\varphi}$,
where $X\in A$-$\mathrm{Mod}$, $Y\in B$-$\mathrm{Mod}$ and $\varphi:X\otimes_{A}M\rightarrow Y$ is a right $B$-homomorphism, and a right $T$-map is identified with a pair $(f_{1}, f_{2}) :(X_{1}, Y_{1})_{\varphi_{1}}\rightarrow(X_{2}, Y_{2})_{\varphi_{2}}$, where $f_{1}:X_{1}\rightarrow X_{2}$ is an $A$-map and $f_{2}:Y_{1}\rightarrow Y_{2}$ a $B$-map, such that $f_{2}\varphi_{1}=\varphi_{2}(f_{1}\otimes1)$. For a right $T$-module $(X, Y)_{\varphi}$, we denote $\widetilde{\varphi} :X\rightarrow \Hom_{B}(M, Y)$ the involution map given by $\widetilde{\varphi}(x)(m)=\varphi( x\otimes m )$ for each $x\in X,~m\in M$.

   Let $\ccc{X_{1}}{X_{2}}{\phi^{X}}$ be a left $T$-module and $(W_{1},W_{2})_{\varphi_{W}}$ be a right $T$-module. By \cite[Proposition 3.6.1]{pa17}, there is an isomorphism of abelian groups
   $$(W_{1},W_{2})_{\varphi_{W}}\oo_{T}\ccc{X_{1}}{X_{2}}{\phi^{X}}=(W_{1}\oo_{A}X_{1}\oplus W_{2}\oo_{B}X_{2})/H,$$
   where $H$ is generated by all elements of the form $\varphi_{W}(w_{1}\oo m)\oo x_{2}-w_{1}\oo\phi^{X}(m\oo x_{2})$ with $w_{1}\in W_{1}, x_{2}\in X_{2}, m\in M$.

\bigskip
\section{Gorenstein modules respect to duality pairs over triangular matrix rings}\label{sec3}
\bigskip

 For two classes $\mathcal{C}$ and $\mathcal{F}$ of modules, we set the following classes of $T$-modules:
\begin{align*}
&\mathfrak{U}^{\mathcal{C}}_{\mathcal{F}}=\{N=\left(\begin{smallmatrix}  N_{1}  \\   N_{2} \\\end{smallmatrix}\right)_{\phi^{N}}\mid N_{1}\in\mathcal{C},~N_{2}\in\mathcal{F} \};\\
&\mathfrak{B}^{\mathcal{C}}_{\mathcal{F}}=\{X=\left(\begin{smallmatrix}  X_{1}  \\   X_{2} \\\end{smallmatrix}\right)_{\phi^{X}}\mid \phi^{X}\text{ is monomorphic},~\mathrm{Coker}\phi^{X}\in\mathcal{C},~X_{2}\in\mathcal{F} \};\\
&\mathfrak{T}^{\mathcal{C}}_{\mathcal{F}}=\{Y=\left(\begin{smallmatrix}  Y_{1}  \\   Y_{2} \\\end{smallmatrix}\right)_{\widetilde{\phi}^{Y}}\mid \widetilde{\phi}^{Y}\text{ is epimorphic},~Y_{1}\in\mathcal{C},~\mathrm{Ker}\widetilde{\phi}^{Y}\in\mathcal{F} \}.
\end{align*}

There are similar symbols $\mathfrak{B}_{\mathcal{C},\mathcal{D}}$, $\mathfrak{T}_{\mathcal{C},\mathcal{D}}$ for the case of right $T$-modules.

L. Mao \cite{ma20} studied symmetric or perfect duality pairs over formal triangular matrix rings.

\begin{lemma} \label{lem3.1}\cite[Lemma 2.3 and Theorem 2.5]{ma20} Let $\mathcal{C}_{1}$ (resp. $\mathcal{C}_{2}$) be a class of left (resp. right) $A$-modules and $\mathcal{F}_{1}$ (resp. $\mathcal{F}_{2}$) be a
class of left (resp. right) $B$-modules. Suppose that  $M_{B}$ is finitely presented and $\mathrm{Tor}^{B}_{1}(M, \mathcal{F}_{1})=0$. Then $(\mathfrak{B}^{\mathcal{C}_{1}}_{\mathcal{F}_{1}}, \mathfrak{T}_{\mathcal{C}_{2},\mathcal{F}_{2}})$ is a complete duality pair if and only if
$(\mathcal{C}_{1},\mathcal{C}_{2})$ and $(\mathcal{F}_{1},\mathcal{F}_{2})$ are complete duality pairs.
\end{lemma}

By the proof of Lemma \ref{lem3.1}, one can get the following result.

\begin{proposition} \label{prop3.2}
Let $\mathcal{C}_{1}$ (resp. $\mathcal{C}_{2}$) be a class of left (resp. right) $A$-modules and $\mathcal{F}_{1}$ (resp. $\mathcal{F}_{2}$) be a
class of left (resp. right) $B$-modules.
 Suppose that  $M_{B}$ is finitely presented.
Then $(\mathfrak{B}^{\mathcal{C}_{1}}_{\mathcal{F}_{1}}, \mathfrak{T}_{\mathcal{C}_{2},\mathcal{F}_{2}})$ is a semi-complete duality pair if and only if
$(\mathcal{C}_{1},\mathcal{C}_{2})$ and $(\mathcal{F}_{1},\mathcal{F}_{2})$ are semi-complete duality pairs.
\end{proposition}

 Let $\mathcal{X}$ be a class of left $R$-modules closed under direct summands and finite direct sums.
Given a natural number $n$ and a left $R$-module $N$, we shall say that $N$ has finite $\mathcal{X}$-projective (resp., $\mathcal{X}$-injective) dimension less than or equal to $n$ if there exists an exact sequence
$$0\rightarrow X_{n}\rightarrow \cdots\rightarrow  X_{1}\rightarrow X_{0}\rightarrow N\rightarrow 0~~~(0\rightarrow N\rightarrow X_{0}\rightarrow X_{1}\rightarrow\cdots\rightarrow X_{n}\rightarrow 0)$$ such that $X_{i}$ belongs to $\mathcal{X}$ for $i=0, 1,\cdots, n$.

\begin{setup}\label{set3.3} Let $\mathcal{C}_{1}$ (resp. $\mathcal{C}_{2}$) be a class of left (resp. right) $A$-modules and $\mathcal{F}_{1}$ (resp. $\mathcal{F}_{2}$) be a
class of left (resp. right) $B$-modules. In this section, we always assume that

$(1)$ $\mathcal{D}_{A}=(\mathcal{C}_{1},\mathcal{C}_{2})$ and $\mathcal{D}_{B}=(\mathcal{F}_{1},\mathcal{F}_{2})$ are semi-complete duality pairs;

$(2)$ $M_{B}$ is finitely presented and has finite $\mathcal{F}_{2}$-injective dimension.

\end{setup}

In the rest of this paper, we denote by $\mathcal{D}_{T}=(\mathfrak{B}^{\mathcal{C}_{1}}_{\mathcal{F}_{1}}, \mathfrak{T}_{\mathcal{C}_{2},\mathcal{F}_{2}})$ the semi-complete duality pair in $T$-Mod induced by semi-complete duality pairs $\mathcal{D}_{A}$ and $\mathcal{D}_{B}$.

\subsection{Gorenstein $\mathcal{D}_{T}$-projective  modules}

\begin{theorem} \label{the3.4} Assume Setup \ref{set3.3} and let  $X=\ccc{X_{1}}{X_{2}}{\phi^{X}}$ be a left $T$-module. Suppose that $\Hom_{B}(M,D)$ has finite $\mathcal{C}_{2}$-injective dimension for each $D\in\mathcal{F}_{2}$. Then the following conditions are equivalent:

$(1)$ $X$ is a Gorenstein $\mathcal{D}_{T}$-projective left $T$-module.

$(2)$ $X_{2}$ is a Gorenstein $\mathcal{D}_{B}$-projective left $B$-module, $\mathrm{Coker}\phi^{X}$ is a Gorenstein $\mathcal{D}_{A}$-projective left $A$-module, and $\phi^{X}$ is a monomorphism.
\end{theorem}
\begin{proof}

By Proposition \ref{prop3.2}, we know that $\mathcal{D}_{T}$ is a semi-complete duality pair if and only if
$\mathcal{D}_{A}$ and $\mathcal{D}_{B}$ are semi-complete duality pairs.

$(1)\Rightarrow (2)$
If  $X$ is a Gorenstein $\mathcal{D}_{T}$-projective left $T$-module, then there is an exact  $\mathfrak{T}_{\mathcal{C}_{2},\mathcal{F}_{2}}^{\otimes}$-acyclic complex of projective left $T$-modules
$$~\mathbf{P}:~\cdots \longrightarrow
\ccc{P^{-1}_{1}}{P^{-1}_{2}}{\phi^{-1}}\stackrel{\binom{\partial^{-1}_{1}}{\partial^{-1}_{2}}{}}
\longrightarrow\ccc{P^{0}_{1}}{P^{0}_{2}}{\phi^{0}}\stackrel{\binom{\partial^{0}_{1}}{\partial^{0}_{2}}}\longrightarrow
\ccc{P^{1}_{1}}{P^{1}_{2}}{\phi^{1}}\stackrel{}\longrightarrow
\cdots$$
with $X=\mathrm{Ker}\binom{\partial^{0}_{1}}{\partial^{0}_{2}}$. By \cite[Theorem 3.1]{ah00}, we get the exact sequence
$$~\mathbf{P}_{2}:~\cdots\rightarrow P_{2}^{-1}\stackrel{\partial^{-1}_{2}}\rightarrow  P_{2}^{0}\stackrel{\partial^{0}_{2}}\rightarrow P_{2}^{1} \rightarrow \cdots $$
of projective left $B$-modules with $X_{2}=\mathrm{Ker}(\partial^{0}_{2})$.
Let $D\in \mathcal{F}_{2}$. Then there is an exact sequence of right $T$-modules
$$0\rightarrow (0, D)\rightarrow (\Hom_{B}(M,D),D) \rightarrow (\Hom_{B}(M,D),0)\rightarrow 0,$$
which induces the exact sequence of complexes
$$0\rightarrow (0, D)\oo_{T}\mathbf{P}\rightarrow (\Hom_{B}(M,D),D)\oo_{T}\mathbf{P} \rightarrow (\Hom_{B}(M,D),0)\oo_{T}\mathbf{P}\rightarrow 0.$$
Since $(\Hom_{B}(M,D),D)$ is belonging to $\mathfrak{T}_{\mathcal{C}_{2},\mathcal{F}_{2}}$, the complex $(\Hom_{B}(M,D),D)\oo_{T}\mathbf{P}$ is exact.
By assumption, $\Hom_{B}(M,D)$ has finite $\mathcal{C}_{2}$-injective dimension, so one can check that $(\Hom_{B}(M,D),0)$ has finite $\mathfrak{T}_{\mathcal{C}_{2},\mathcal{F}_{2}}$-injective dimension and the complex $(\Hom_{B}(M,D),0)\oo_{T}\mathbf{F}$ is exact.
It follows that $D\oo_{B}\mathbf{P}_{2}=(0, D)\oo_{T}\mathbf{P}$ is exact. Whence $X_{2}$ is a Gorenstein $\mathcal{D}_{B}$-projective left $B$-module.

Let $\iota_{1}:X_{1}\rightarrow P^{0}_{1}$ and $\iota_{2}:X_{2}\rightarrow P^{0}_{2}$ be the inclusions. Consider the following commutative diagram in $A$-Mod:
$$\xymatrix{
  M\otimes_{B} X_{2} \ar[d]_{\phi^{X}} \ar[r]^{1_{M}\otimes_{B} \iota_{2}} & M\otimes_{B} P^{0}_{2} \ar[d]_{ \phi^{P^{0}}} \\
  X_{1} \ar[r]^{\iota_{1}} & P^{0}_{1}   }
$$
By assumption that $M_{B}$ has finite $\mathcal{F}_{2}$-injective dimension, one can check that $M\oo_{B}\mathbf{P}_{2}$ is exact. Thus $1_{M}\otimes_{B} \iota_{2}$ is a monomorphism. Also $\phi^{P^{0}}$ is a monomorphism by \cite[Theorem 3.1]{ah00}. So $\phi^{X}$ is a monomorphism by the commutative diagram above.

For any $i\in \mathbb{Z}$, there exists $\overline{\partial^{i}_{1}}:P^{i}_{1}/\mathrm{Im}(\phi^{i}) \rightarrow P^{i+1}_{1}/\mathrm{Im}(\phi^{i+1}) $ such that the following diagram with exact rows is commutative.
$$\xymatrix{
   & \vdots \ar[d]_{}  & \vdots \ar[d]_{}  & \vdots \ar@{..>}[d]_{}  &  \\
  0  \ar[r]^{} & M\oo_{B}P^{-1}_{2} \ar[d]_{1\oo\partial^{-1}_{2}} \ar[r]^-{\phi^{-1}} & P^{-1}_{1} \ar[d]_{\partial^{-1}_{1}} \ar[r]^{} & P^{-1}_{1}/\mathrm{Im}(\phi^{-1})  \ar@{..>}[d]_{\overline{\partial^{-1}_{1}}} \ar[r]^{} & 0  \\
   0  \ar[r]^{} & M\oo_{B}P^{0}_{2} \ar[d]_{1\oo\partial^{0}_{2}} \ar[r]^-{\phi^{0}} & P^{0}_{1} \ar[d]_{\partial^{0}_{1}} \ar[r]^{} & P^{0}_{1}/\mathrm{Im}(\phi^{0})  \ar@{..>}[d]_{\overline{\partial^{0}_{1}}} \ar[r]^{} & 0  \\
  0  \ar[r]^{} & M\oo_{B}P^{1}_{2} \ar[d]_{} \ar[r]^-{\phi^{1}} & P^{1}_{1} \ar[d]_{} \ar[r]^{} & P^{1}_{1}/\mathrm{Im}(\phi^{1})  \ar@{..>}[d]_{} \ar[r]^{} & 0  \\
   & \vdots & \vdots & \vdots&   }$$
Since the first column and the second column are exact, we get the exact sequence of projective left $A$-modules
$$~\overline{\mathbf{P}_{1}}:~\cdots\rightarrow P^{-1}_{1}/\mathrm{Im}(\phi^{-1})\stackrel{\overline{\partial^{-1}_{1}}}\rightarrow  P^{0}_{1}/\mathrm{Im}(\phi^{0})\stackrel{\overline{\partial^{0}_{1}}}\rightarrow P^{1}_{1}/\mathrm{Im}(\phi^{1}) \rightarrow \cdots $$
with $X_{1}/\mathrm{Im}(\phi^{X}) \cong \mathrm{Ker}(\overline{\partial^{0}_{1}})$.

Let $G\in\mathcal{C}_{2}$. Then each exact sequence of left $A$-modules
$$\xymatrix{0\ar[r] &M\oo_{B}P^{i}_{2}  \ar[r]^-{\phi^{i}} & P^{i}_{1}  \ar[r]^{} & P^{i}_{1}/\mathrm{Im}(\phi^{i})  \ar[r]^{} & 0}$$
induces the exact sequence
$$\xymatrix{ G\oo_{A}M\oo_{B}P^{i}_{2}  \ar[r]^-{1\oo\phi^{i}} & G\oo_{A}P^{i}_{1}  \ar[r]^{} & G\oo_{A}(P^{i}_{1}/\mathrm{Im}(\phi^{i}))  \ar[r]^{} & 0}$$
So we have
$$G\oo_{A}(P^{i}_{1}/\mathrm{Im}(\phi^{i}))\cong G\oo_{A}P^{i}_{1}/\mathrm{Im}(1\oo\phi^{i})\cong (G,0)\oo_{T}\ccc{P^{i}_{1}}{P^{i}_{2}}{\phi^{i}}.$$

Since $(G,0)\in \mathfrak{T}_{\mathcal{C}_{2},\mathcal{F}_{2}}$, $G\oo_{A}\overline{\mathbf{P}_{1}}\cong(G, 0)\oo_{T}\mathbf{P}$ is exact. So $\mathrm{Coker}\phi^{X}$ is a Gorenstein $\mathcal{D}_{A}$-projective left $A$-module.

$(2)\Rightarrow(1)$ Since $\phi^{X}$ is a monomorphism, there exists an exact sequence in $T$-Mod
$$0\rightarrow\left(\begin{matrix}  M\otimes_{B}X_{2}  \\ X_{2}  \\\end{matrix}\right)_{id}\rightarrow\left(\begin{matrix}  X_{1}  \\ X_{2} \\\end{matrix}\right)_{\phi^{X}}\rightarrow\left(\begin{matrix}  \mathrm{coker}\phi^{X}  \\ 0 \\\end{matrix}\right)_{0}\rightarrow0.
$$
By the dual of \cite[Lemma 3.5]{gj21}, one can check that the class of  Gorenstein $\mathcal{D}_{T}$-projective left $T$-module is  closed under extensions. So we only need to verify that $\left(\begin{smallmatrix}  M\otimes_{B}X_{2}  \\ X_{2}  \\\end{smallmatrix}\right)$ and $\left(\begin{smallmatrix}  \mathrm{coker}\phi^{X}  \\ 0 \\\end{smallmatrix}\right)$ are  Gorenstein $\mathcal{D}_{T}$-projective.

We first prove that $\left(\begin{smallmatrix}  M\otimes_{B}X_{2}  \\ X_{2}  \\\end{smallmatrix}\right)$ is a Gorenstein $\mathcal{D}_{T}$-projective module. In fact, there is an exact   $\mathcal{F}_{2}$-acyclic complex
$$~\mathbf{U}:~\cdots\rightarrow Q^{-1}\stackrel{\partial^{-1}}\rightarrow  Q^{0}\stackrel{\partial^{0}}\rightarrow Q^{1} \rightarrow \cdots $$
of projective left $B$-modules with $X_{2}=\mathrm{Ker}(\partial^{0}_{2})$. Since $M_{B}$ has finite $\mathcal{F}_{2}$-injective dimension, one can check that $M\oo_{B}\mathbf{U}$ is exact.  So we get the exact sequence of projective left $T$-modules
$$~\mathbf{V}:~\cdots \longrightarrow
\ccc{M\oo_{B}Q^{-1}}{Q^{-1}}{}\stackrel{\binom{1\oo\partial^{-1}}{\partial^{-1}}{}}
\longrightarrow\ccc{M\oo_{B}Q^{0}}{Q^{0}}{}\stackrel{\binom{1\oo\partial^{0}}{\partial^{0}}}\longrightarrow
\ccc{M\oo_{B}Q^{1}}{Q^{1}}{}\stackrel{}\longrightarrow
\cdots$$
with $\ccc{ M\otimes_{B}X_{2}}{ X_{2}}{}\cong\mathrm{Ker}\binom{1\oo\partial^{0}}{\partial^{0}}$. For any  right $T$-module $(H_{1}, H_{2})\in\mathfrak{T}_{\mathcal{C}_{2},\mathcal{F}_{2}}$, there exists an exact sequence in Mod-$T$
$$0\rightarrow (0, H_{2})\rightarrow (H_{1}, H_{2}) \rightarrow (H_{1},0)\rightarrow 0.$$
Since each $\ccc{M\oo_{B}P^{i}}{P^{i}}{}$ is a projective left $T$-module, we get the exact sequence
$$0\rightarrow (0, H_{2})\oo_{T}\ccc{M\oo_{B}P^{i}}{P^{i}}{}\rightarrow (H_{1}, H_{2})\oo_{T}\ccc{M\oo_{B}P^{i}}{P^{i}}{} \rightarrow (H_{1},0)\oo_{T}\ccc{M\oo_{B}P^{i}}{P^{i}}{}\rightarrow 0.$$
Note that $(H_{1},0)\oo_{T}\ccc{M\oo_{B}P^{i}}{P^{i}}{}\cong (H_{1}\oo_{A}M\oo_{B}P^{i})/(H_{1}\oo_{A}M\oo_{B}P^{i})=0$. Thus $(H_{1}, H_{2})\oo_{T}\ccc{M\oo_{B}P^{i}}{P^{i}}{}\cong(0, H_{2})\oo_{T}\ccc{M\oo_{B}P^{i}}{P^{i}}{}$. So $(H_{1}, H_{2})\oo_{T}\mathbf{V}\cong(0, H_{2})\oo_{T}\mathbf{V}\cong H_{2}\oo_{B}\mathbf{U}$ is exact since $H_{2}\in \mathcal{F}_{2}$. Hence $\left(\begin{smallmatrix}  M\otimes_{B}X_{2}  \\ X_{2}  \\\end{smallmatrix}\right)$ is a Gorenstein $\mathcal{D}_{T}$-projective left $T$-module.

Next we prove that $\ccc{\mathrm{Coker}\phi^{X}}{0}{}$ is a Gorenstein $\mathcal{D}_{T}$-projective left $T$-module. There is an exact $\mathcal{C}_{2}^{\otimes}$-acyclic sequence
$$~\mathbf{L}:~\cdots\rightarrow L^{-1}\stackrel{d^{-1}}\rightarrow  L^{0}\stackrel{d^{0}}\rightarrow L^{1} \rightarrow \cdots $$
of projective left $A$-modules with $\mathrm{coker}\phi^{X}=X_{1}/\mathrm{Im}(\phi^{X})=\mathrm{Ker}(d^{0})$. Then we get the exact sequence of projective left $T$-modules
$$~\ccc{\mathbf{L}}{0}{}:~\cdots \longrightarrow
\ccc{L^{-1}}{0}{}\stackrel{\binom{d^{-1}}{0}{}}
\longrightarrow\ccc{L^{0}}{0}{}\stackrel{\binom{d^{0}}{0}}\longrightarrow
\ccc{L^{1}}{0}{}\stackrel{}\longrightarrow
\cdots$$
such that $\ccc{\mathrm{Coker}\phi^{X}}{0}{}=\mathrm{Ker}\binom{d^{0}}{0}$.

Let $(H_{1}, H_{2})_{\widetilde{\varphi_{H}}}\in\mathfrak{T}_{\mathcal{C}_{2},\mathcal{F}_{2}}$. Then there exists an exact sequence
 $$0\rightarrow \mathrm{Ker}(\widetilde{\varphi_{H}})\rightarrow H_{1}\rightarrow \Hom_{B}(M,H_{2})\rightarrow 0$$
with $H_{2}\in\mathcal{F}_{2}$ and $ \mathrm{Ker}(\widetilde{\varphi_{H}})\in\mathcal{C}_{2}$.
Since $\mathbf{L}$ is a complex consisting of projective modules, we have a short exact sequence of complexes
 $$0\rightarrow \mathrm{Ker}(\widetilde{\varphi_{H}})\oo_{A}\mathbf{L}\rightarrow H_{1}\oo_{A}\mathbf{L}\rightarrow \Hom_{B}(M,H_{2})\oo_{A}\mathbf{L}\rightarrow 0.$$
By hypotheses,  $\Hom_{B}(M,H_{2})$ has finite $\mathcal{C}_{2}$-injective dimension. It follows that  the complexes $\mathrm{Ker}(\widetilde{\varphi_{H}})\oo_{A}\mathbf{L}$ and $\Hom_{B}(M,H_{2})\oo_{A}\mathbf{L}$ are exact. Therefore, the complex $(H_{1}, H_{2})\oo_{T}\ccc{\mathbf{L}}{0}{}\cong H_{1}\oo_{A}\mathbf{L}$ is exact, so $\ccc{\mathrm{Coker}\phi^{X}}{0}{}$ is a Gorenstein $\mathcal{D}_{T}$-projective $T$-module.
\end{proof}

By \cite[Theorem A6]{br14}, for any ring $R$ and a symmetric duality pair $\mathcal{D}_{R}$, the class of Gorenstein $\mathcal{D}_{R}$-projective modules and the class of $\mathcal{D}_{R}$-Gorenstein projective modules are coincide. We have the following characterization.

\begin{corollary} \label{th3.5}Assume Setup \ref{set3.3} and let  $X=\ccc{X_{1}}{X_{2}}{\phi^{X}}$ be a left $T$-module. Suppose that $\Hom_{B}(M,D)$ has finite $\mathcal{C}_{2}$-injective dimension for each $D\in\mathcal{F}_{2}$. Then the following conditions are equivalent:

$(1)$ $X$ is a $\mathcal{D}_{T}$-Gorenstein projective left $T$-module.

$(2)$ $X_{2}$ is a $\mathcal{D}_{B}$-Gorenstein projective left $B$-module, $\mathrm{Coker}\phi^{X}$ is a $\mathcal{D}_{A}$-Gorenstein projective left $A$-module, and $\phi^{X}$ is a monomorphism.
\end{corollary}

Let $R$ be any ring. The class of all flat left $R$-modules and  injective right $R$-modules is denoted  by $\mathcal{I}_{R}$  and $_{R}\mathcal{F}$, respectively.
Using results from \cite{Pre09} it is shown very succinctly in \cite[Lemmas 5.5-5.7]{CS21}, that we have a semi-complete duality pair $ (\langle _{R}\mathcal{F}\rangle, \langle \mathcal{I}_{R}\rangle)$
where $\langle _{R}\mathcal{F}\rangle$ is the definable class (meaning it is closed under products, direct
limits, and pure submodules) generated by the class of all flat left $R$-modules and $\langle \mathcal{I}_{R}\rangle$
is the definable class generated by the class of all injective right $R$-modules.

Recently, in order to prove that Gorenstein flat modules are always closed under extensions over any ring, \v{S}aroch and \v{S}\v{t}ov\'{\i}\v{c}ek \cite{js18} introduced the notion of PGF-modules. Recall that a \emph{projectively coresolved Gorenstein flat module}, or a PGF-module for short, is a syzygy module in an acyclic complex
$$\cdots\rightarrow P_{-1}\rightarrow  P_{0}\rightarrow P_{1} \rightarrow P_{2}\rightarrow\cdots $$
consisting of projective modules which remains exact after tensoring by an arbitrary injective
 left $R$-module. Denote by $\mathcal{PGF}_{R}$ the class of PGF-left $R$-modules. Let $\mathcal{D}_{A}=(\langle _{A}\mathcal{F}\rangle, \langle \mathcal{I}_{A}\rangle)$  and $\mathcal{D}_{B}=(\langle _{B}\mathcal{F}\rangle, \langle \mathcal{I}_{B}\rangle)$, respectively. By \cite[Corollary 2.12]{gj21} or
 \cite[Theorem 3.4]{js18}, an $R$-module is $\mathrm{PGF}$  if and only if it is a Gorenstein $(\langle _{R}\mathcal{F}\rangle, \langle \mathcal{I}_{R}\rangle)$-projective module.
Then we have the following result.

\begin{corollary} \label{cor3.6} Let  $X=\ccc{X_{1}}{X_{2}}{\phi^{X}}$ be a left $T$-module. Suppose that $M_{B}$ has finite injective dimension, $_{A}M$ has finite flat dimension. Then the following conditions are equivalent:

$(1)$ $X$ is a $\mathrm{PGF}$ left $T$-module.

$(2)$ $X_{2}$ is a $\mathrm{PGF}$ left $B$-module, $\mathrm{Coker}\phi^{X}$ is a $\mathrm{PGF}$ left $A$-module, and $\phi^{X}$ is a monomorphism.

Moreover, if $M_{B}$ is  finitely presented, then the above conditions are also equivalent to

$(3)$ $X$ is a Gorenstein $(\mathfrak{B}^{\langle _{A}\mathcal{F}\rangle}_{\langle _{B}\mathcal{F}\rangle}, \mathfrak{T}_{\langle \mathcal{I}_{A}\rangle,\langle \mathcal{I}_{B}\rangle})$-projective left $T$-module.

\end{corollary}
\begin{proof}The equivalence of (1) and (2) follows by \cite[Theorem 2.8]{wu21}.
If a module $N$ is  finitely presented as a right $B$-module and  flat as a left $A$-module, then by \cite[Lemma 2.2]{mao20}, one can check that the functor
$\mathrm{Hom}_{B}(N,-)$ preserves injectives, products, colimits and pure embedding. Therefore, we obtain that
$\mathrm{Hom}_{B}(N,\langle \mathcal{I}_{B}\rangle)\subseteq \langle \mathcal{I}_{A}\rangle$.
It follows that if $_{A}M$ has finite flat dimension, then the $\langle \mathcal{I}_{A}\rangle$-injective dimension of $\Hom_{B}(M,E)$ is finite for each module $E\in\langle  \mathcal{I}_{B}\rangle$. Moreover, since $M_{B}$ has finite injective dimension, $M_{B}$ has finite $\langle \mathcal{I}_{B}\rangle$-injective dimension. Therefore, all the assumptions of Setup \ref{set3.3} and Theorem \ref{the3.4} are satisfied. Thus the equivalence of (2) and (3) follows.
\end{proof}

\begin{corollary} \label{cor3.7} Suppose that $M_{B}$ is finitely presented and has finite injective dimension, $_{A}M$ has finite flat dimension. Then the duality pairs $(\mathfrak{B}^{\langle _{A}\mathcal{F}\rangle}_{\langle _{B}\mathcal{F}\rangle}, \mathfrak{T}_{\langle \mathcal{I}_{A}\rangle,\langle \mathcal{I}_{B}\rangle})$ and $(\langle _{T}\mathcal{F}\rangle, \langle\mathcal{I}_{T}\rangle)$ are coincide.
\end{corollary}
\begin{proof}By \cite[Corollary 2.12]{gj21}, a $T$-module is $\mathrm{PGF}$  if and only if it is  Gorenstein-$(\langle _{T}\mathcal{F}\rangle, \langle\mathcal{I}_{T}\rangle)$ projective. Therefore,
Corollary \ref{cor3.6} tells us that the class of Gorenstein $(\mathfrak{B}^{\langle _{A}\mathcal{F}\rangle}_{\langle _{B}\mathcal{F}\rangle}, \mathfrak{T}_{\langle \mathcal{I}_{A}\rangle,\langle \mathcal{I}_{B}\rangle})$-projective modules and the class of Gorenstein-$(\langle _{T}\mathcal{F}\rangle, \langle\mathcal{I}_{T}\rangle)$ projective modules are the same.
Consider the $\mathrm{PGF}$ module $T$. It is a syzygy module
in complex $\cdots\rightarrow 0\rightarrow T \rightarrow T \rightarrow 0\rightarrow\cdots $
which remains exact after applying functor $X\otimes_{T}-$ and $Y\otimes_{T}-$ for any $X\in \mathfrak{T}_{\langle \mathcal{I}_{A}\rangle,\langle \mathcal{I}_{B}\rangle}$ and $Y\in \langle\mathcal{I}_{T}\rangle$. This implies that $\mathfrak{T}_{\langle \mathcal{I}_{A}\rangle,\langle \mathcal{I}_{B}\rangle}=\langle\mathcal{I}_{T}\rangle$. Thus these two duality pairs are the same.
\end{proof}

\bigskip
\subsection{Gorenstein $\mathcal{D}_{T}$-injective  modules}

\begin{theorem} \label{the3.7} Assume Setup \ref{set3.3} and let  $X=({X_{1}},{X_{2}})_{\widetilde{\varphi^{X}}}$ be a right $T$-module.  Suppose that $\Hom_{B}(M,D)$ has finite $\mathcal{C}_{2}$-injective dimension for each $D\in\mathcal{F}_{2}$. Then the following conditions are equivalent:

$(1)$ $X$ is a Gorenstein $\mathcal{D}_{T}$-injective right $T$-module.

$(2)$ $X_{2}$ is a Gorenstein $\mathcal{D}_{A}$-injective right $B$-module, $\mathrm{Ker}(\widetilde{\varphi^{X}})$ is a Gorenstein $\mathcal{D}_{B}$-injective right $A$-module, and $\widetilde{\varphi^{X}}$ is an epimorphism.
\end{theorem}
\begin{proof} By \cite[Proposition 5.1]{ah99}, a right $T$-module $X=({X_{1}},{X_{2}})_{\widetilde{\varphi^{X}}}$ is injective if and only if $X_{2}$ is an injective right $B$-module, $\mathrm{Ker}(\widetilde{\varphi^{X}})$ is an  injective right $A$-module, and $\widetilde{\varphi^{X}}$ is an epimorphism. Then the proof is dual to that of Theorem \ref{the3.4}.
\end{proof}

Recall that a left $R$-module $X$ is \emph{FP-injective} or \emph{absolutely pure} if
$\mathrm{Ext}^{1}_{R}(N,X)=0$ for every finitely presented left $R$-module $N$.  Let $\mathcal{FI}_{R}$ be the class of all absolutely pure modules.  As in \cite{gi10}, an right $R$-module $M$ \emph{Ding injective} if there exists an exact
complex of injectives
$$\cdots\rightarrow I_{1} \rightarrow I_{0} \rightarrow I^{0} \rightarrow I^{1} \rightarrow \cdots$$
with $M = \mathrm{Ker} (I^{0}\rightarrow I^{1})$ and which remains exact after applying $\Hom_{R}(E,-)$ for any absolutely pure module $E$. Denote by $\mathcal{DI}_{R}$ the class of all Ding injective right $R$-modules.  Let $\mathcal{D}_{A}=(\langle _{A}\mathcal{F}\rangle, \langle \mathcal{I}_{A}\rangle)$  and $\mathcal{D}_{B}=(\langle _{B}\mathcal{F}\rangle, \langle \mathcal{I}_{B}\rangle)$, respectively.
 By \cite[Proposition 2.11]{gj21}, an $R$-module is Ding injective  if and only if it is a Gorenstein $(\langle _{R}\mathcal{F}\rangle, \langle \mathcal{I}_{R}\rangle)$-injective module.
Then we have the following result. One can compare it with \cite[Theorem 4.4]{mao22}.

\begin{corollary} \label{cor3.8}Let  $X=({X_{1}},{X_{2}})_{\widetilde{\varphi^{X}}}$ be a right $T$-module. Suppose that $M_{B}$ is  finitely presented and has finite FP-injective dimension, $_{A}M$ has finite flat dimension. Then the following conditions are equivalent:

$(1)$ $X$ is a Gorenstein $(\mathfrak{B}^{\langle _{A}\mathcal{F}\rangle}_{\langle _{B}\mathcal{F}\rangle}, \mathfrak{T}_{\langle \mathcal{I}_{A}\rangle,\langle \mathcal{I}_{B}\rangle})$-injective right $T$-module.

$(2)$ $X_{2}$ is a Ding injective right $B$-module, $\mathrm{Ker}(\widetilde{\varphi^{X}})$ is a Ding injective right $A$-module, and $\widetilde{\varphi^{X}}$ is an epimorphism.

Moreover, if $M_{B}$ has finite injective dimension, then the above conditions are also equivalent to

$(3)$ $X$ is a Ding injective right $T$-module.
\end{corollary}
\begin{proof}
By the proof of Corollary \ref{cor4.4}, we know that if $_{A}M$ has finite flat dimension, then the $\langle \mathcal{I}_{A}\rangle$-injective dimension of $\Hom_{B}(M,E)$ is finite for each module $E\in \langle \mathcal{I}_{B}\rangle$.
Moreover, since $M_{B}$ has finite FP-injective dimension and $\mathcal{FI}_{B}\subseteq \langle \mathcal{I}_{B}\rangle$, $M_{B}$ has finite $\langle \mathcal{I}_{B}\rangle$-injective dimension.
Then all the assumptions of Setup \ref{set3.3} and Theorem \ref{the3.7} are satisfied. Thus the equivalence of (1) and (2) follows. Moreover, if $M_{B}$ has finite injective dimension, then the equivalence of (2) and (3) follows by
Corollary \ref{cor3.7}. \end{proof}

\bigskip

\subsection{Gorenstein $\mathcal{D}_{T}$-flat modules}

\begin{lemma}\label{lem3.11} \cite[Proposition 4.6 or  Corollary 5.3]{gj21} If $\mathcal{D}_{R}=\mathcal{(L,A)}$ is a semi-complete duality  pair, then the class of all Gorenstein $\mathcal{D}_{R}$-flat modules is closed under extensions.
\end{lemma}

\begin{theorem} \label{the3.8} Assume Setup \ref{set3.3} and let  $X=\ccc{X_{1}}{X_{2}}{\phi^{X}}$ be a left $T$-module. Suppose that $\Hom_{B}(M,D)$ has finite $\mathcal{C}_{2}$-injective dimension for each $D\in\mathcal{F}_{2}$. Then the following conditions are equivalent:

$(1)$ $X$ is a Gorenstein $\mathcal{D}_{T}$-flat left $T$-module.

$(2)$ $X_{2}$ is a Gorenstein $\mathcal{D}_{B}$-flat left $B$-module, $\mathrm{Coker}\phi^{X}$ is a Gorenstein $\mathcal{D}_{A}$-flat left $A$-module, and $\phi^{X}$ is a monomorphism.
\end{theorem}

\begin{proof}
From Lemma \ref{lem3.11}, we know that the class of  Gorenstein $\mathcal{D}_{T}$-flat left $T$-module is  closed under extensions. By \cite[Proposition 1.14]{fo75}, we know that $\ccc{F_{1}}{F_{2}}{\phi^{F}}$ is flat if and only if $F_{2}$ is flat in $B$-Mod, $\mathrm{coker}\phi^{F}$ is flat in $A$-Mod and $\phi^{F}$ is monomorphic. Then the proof follows
by an argument similar to that in Theorem \ref{the3.4}.
\end{proof}

\bigskip

Recall that a module $M$ is called\emph{ Gorenstein flat} \cite{ee00} if there exists an exact sequence
$$\cdots\rightarrow F_{1}\rightarrow F_{0}\rightarrow F_{-1}\rightarrow\cdots$$
of flat modules such that $M\cong \mathrm{Ker}(F_{0}\rightarrow F_{-1})$ and $I\otimes-$ leaves the sequence exact whenever $I$ is an injective right module.
Let $\mathcal{D}_{A}=(\langle _{A}\mathcal{F}\rangle, \langle \mathcal{I}_{A}\rangle)$  and $\mathcal{D}_{B}=(\langle _{B}\mathcal{F}\rangle, \langle \mathcal{I}_{B}\rangle)$, respectively. By \cite[Proposition 2.11]{gj21}, an $R$-module is Gorenstein flat  if and only if it is a Gorenstein $(\langle _{R}\mathcal{F}\rangle, \langle \mathcal{I}_{R}\rangle)$-flat module.
Then we have the following result. One can compare it with \cite[Theorem 2.8]{mao20}

\begin{corollary} \label{cor3.11} Let  $X=\ccc{X_{1}}{X_{2}}{\phi^{X}}$ be a left $T$-module. Suppose that $M_{B}$ has finite injective dimension, $_{A}M$ has finite flat dimension. Then the following conditions are equivalent:

$(1)$ $X$ is a Gorenstein flat left $T$-module.

$(2)$ $X_{2}$ is a Gorenstein flat left $B$-module, $\mathrm{Coker}\phi^{X}$ is a Gorenstein flat left $A$-module, and $\phi^{X}$ is a monomorphism.

Moreover, if $M_{B}$ is  finitely presented, then the above conditions are also equivalent to

$(3)$ $X$ is a Gorenstein $(\mathfrak{B}^{\langle _{A}\mathcal{F}\rangle}_{\langle _{B}\mathcal{F}\rangle}, \mathfrak{T}_{\langle \mathcal{I}_{A}\rangle,\langle \mathcal{I}_{B}\rangle})$-projective left $T$-module.

\end{corollary}
\begin{proof} The equivalence of (1) and (2) follows by  the proof of  \cite[Lemma 2.2]{mao20} and the fact that Gorenstein flat modules are closed under extensions over any ring.
Moreover, if $M_{B}$ is  finitely presented, by an argument similar to Theorem \ref{the3.4},
we see that all the assumptions of Setup \ref{set3.3} and Theorem \ref{the3.8} are satisfied. Thus the result follows.
\end{proof}

\bigskip
\section{ Recollements of stable categories  relative to  triangular matrix rings}\label{sec4}

Let $\mathcal{D}_{R} = \mathcal{(L,A)}$ denote a semi-complete duality pair over a ring $R$. There are following
three model structures induced by Gorenstein modules respect to $\mathcal{D}_{R}$. A nice introduction to the basic idea of a model category can be found in \cite{hv99}.

\begin{lemma} \cite[Corollary 5.1]{gj21}\label{lem4.1} The following abelian model structures are induced by $\mathcal{D}_{R} = \mathcal{(L,A)}$.

(1) The \textbf{Gorenstein $\mathcal{D}_{R}$-projective model structure} exists on $R$-$\mathrm{Mod}$. It is
a cofibrantly generated projective abelian model structure whose cofibrant objects are the
 Gorenstein $\mathcal{D}_{R}$-projective left $R$-modules.

(2) The \textbf{Gorenstein $\mathcal{D}_{R}$-injective model structure} exists on $\mathrm{Mod}$-$R$. It is a cofibrantly generated injective abelian model structure whose fibrant objects are the Gorenstein $\mathcal{D}_{R}$-injective right $R$-modules.

(3) The \textbf{Gorenstein $\mathcal{D}_{R}$-flat model structure} exists on $R$-$\mathrm{Mod}$. It is a cofibrantly generated abelian model structure whose cofibrant objects (resp. trivially cofibrant objects) are the Gorenstein $\mathcal{D}_{R}$-flat left modules (resp. flat left modules). Moreover, the trivial objects in this model structure coincide with those in the Gorenstein $\mathcal{D}_{R}$-projective model structure.
\end{lemma}

According to the recollement constructed by \cite{cp14} and \cite{zh13}, we have the following recollement of abelian categories:
$$\xymatrix{A\text{-}\mathrm{Mod}\ar[r]^{i_{\ast}}&\ar@<-3ex>[l]_{i^{\ast}}
\ar@<3ex>[l]^{i^{!}}T\text{-}\mathrm{Mod}
\ar[r]^{j^{\ast}}&\ar@<-3ex>[l]_{j_{!}}\ar@<3ex>[l]^{j_{\ast}}B\text{-}\mathrm{Mod},}\eqno(4.1)$$
where $i^{\ast}$ is given by $\left(\begin{smallmatrix}  X  \\   Y \\\end{smallmatrix}\right)_{\phi}\mapsto \mathrm{coker}\phi$; $i_{\ast}$ is given by $X\mapsto \left(\begin{smallmatrix}  X  \\   0 \\\end{smallmatrix}\right)$; $i^{!}$ is given by $\left(\begin{smallmatrix}  X  \\   Y \\\end{smallmatrix}\right)_{\phi}\mapsto X$; $j_{!}$ is given by $Y\mapsto \left(\begin{smallmatrix}  M\otimes _{B}Y  \\   Y \\\end{smallmatrix}\right)_{id}$; $j^{\ast}$ is given by $\left(\begin{smallmatrix}  X  \\   Y \\\end{smallmatrix}\right)_{\phi}\mapsto Y$; $j_{\ast}$ is given by $Y\mapsto \left(\begin{smallmatrix}  0  \\   Y \\\end{smallmatrix}\right)$.

If we are given two cofibrantly generated model structures $\mathcal{M}(A)=(_{A}\mathcal{A},{_{A}\mathcal{W}},{_{A}\mathcal{L}})$ and $\mathcal{M}(B)=(_{B}\widetilde{\mathcal{A}},{_{B}\widetilde{\mathcal{W}}},{_{B}\widetilde{\mathcal{L}}})$ on $A$-Mod and $B$-Mod respectively, we investigate in \cite{zhu20} when there exists a cofibrantly generated model structure $\mathcal{M}(T)$ on $T$-Mod and a recollement of $\mathrm{Ho}(\mathcal{M}(T))$ relative to $\mathrm{Ho}(\mathcal{M}(A))$ and $\mathrm{Ho}(\mathcal{M}(B))$.
We defined in  \cite{zhu20} that a bimodule $_{A}M_{B}$ is \emph{perfect relative to $\mathcal{M}(A)$ and $\mathcal{M}(B)$}, if $\mathfrak{B}^{_{A}\mathcal{A}}_{_{B}\widetilde{\mathcal{A}}}\cap \mathfrak{U}_{_{B}\widetilde{\mathcal{W}}\cap {_{B}\widetilde{\mathcal{L}}}}^{_{A}\mathcal{W}\cap {_{A}\mathcal{L}}}=\mathfrak{B}^{_{A}\mathcal{A}\cap {_{A}\mathcal{W}}}_{_{B}\widetilde{\mathcal{A}}\cap {_{B}\widetilde{\mathcal{W}}}}\cap \mathfrak{U}_{_{B}\widetilde{\mathcal{L}}}^{_{A}\mathcal{L}}$.
More precisely, we have

\begin{lemma}\cite[Theorem 4.6]{zhu20} \label{the4.6}Let $T=\left(\begin{smallmatrix}  A & M \\  0 & B \\\end{smallmatrix}\right)$ be an upper triangular matrix ring, and let $\mathcal{M}(A)=(_{A}\mathcal{A},{_{A}\mathcal{W}},{_{A}\mathcal{L}})$ and $\mathcal{M}(B)=(_{B}\widetilde{\mathcal{A}},{_{B}\widetilde{\mathcal{W}}},{_{B}\widetilde{\mathcal{L}}})$ be cofibrantly generated abelian model structures on $A$-$\mathrm{Mod}$ and $B$-$\mathrm{Mod}$, respectively. If $\mathrm{Tor}_{1}^{B}(M,X)=0$ for any $X\in {_{B}\widetilde{\mathcal{A}}}$ and $M$ is perfect relative to $\mathcal{M}(A)$ and $\mathcal{M}(B)$, then $\mathcal{M}(T)=(\mathfrak{B}^{\mathcal{A}_{A}}_{\widetilde{\mathcal{A}}_{B}},{_{T}\mathcal{W}},\mathfrak{U}_{\widetilde{\mathcal{L}}_{B}}^{\mathcal{L}_{A}})$ is a cofibrantly generated abelian model structure on $T$-{\rm Mod} and we have a recollement as shown below
  $$\xymatrixcolsep{4pc}\xymatrix{\mathrm{Ho}(\mathcal{M}(A))
  \ar[r]^{\mathrm{L~i_{\ast}~}\cong\mathrm{~R~i_{\ast}}}&
  \ar@<-4ex>[l]_{\mathrm{L~i^{\ast}}}\ar@<3ex>[l]^{\mathrm{R~i^{!}}}\mathrm{Ho}(\mathcal{M}(T))
\ar[r]^{\mathrm{L~j^{\ast}~}\cong\mathrm{~R~j^{\ast}}}&\ar@<-4ex>[l]_{\mathrm{L~j_{!}}}\ar@<3ex>[l]^{\mathrm{R~j_{\ast}}}
\mathrm{Ho}(\mathcal{M}(B)),}$$
where $\mathrm{L~i^{\ast}}$, $\mathrm{L~i_{\ast}}$, $\mathrm{L~j_{!}}$, $\mathrm{L~j^{\ast}}$, $\mathrm{R~i_{\ast}}$,  $\mathrm{R~j^{\ast}}$, $\mathrm{R~i^{!}}$ and $\mathrm{R~j_{\ast}}$ are the total derived functors of those in (4.1).
\end{lemma}
Let $\mathcal{M}'_{A}=(\mathcal{A}_{A},\mathcal{W}_{A},\mathcal{L}_{A})$ and $\mathcal{M}'_{B}=(\widetilde{\mathcal{A}}_{B},\widetilde{\mathcal{W}}_{B},\widetilde{\mathcal{L}}_{B})$ be two cofibrantly generated model structures on Mod-$A$ and Mod-$B$ respectively.
We define that a bimodule $_{A}M_{B}$ is \emph{coperfect relative to $\mathcal{M}'_{A}$ and $\mathcal{M}'_{B}$}, if $\mathfrak{U}_{\mathcal{A}_{A},\widetilde{\mathcal{A}}_{B}}\cap \mathfrak{T}_{\mathcal{W}_{A}\cap \mathcal{L}_{A},\widetilde{\mathcal{W}}_{B}\cap \widetilde{\mathcal{L}}_{B}}=
\mathfrak{U}_{\mathcal{A}_{A}\cap \mathcal{W}_{A},\widetilde{\mathcal{A}}_{B}\cap \widetilde{\mathcal{W}}_{B}}\cap \mathfrak{T}_{\mathcal{L}_{A},\widetilde{\mathcal{L}}_{B}}$.
Using \cite[Theorem 5.6]{mlx20}, \cite[Remark 4.7]{zhu20} and proceeding in a similar way as in Lemma \ref{the4.6}, we have the following result.

\begin{lemma}\label{lem4.3}If  $\mathrm{Ext}_{B}^{1}(M,F)=0$ for any $F\in \widetilde{\mathcal{L}_{B}}$, and $M$ is coperfect relative to $\mathcal{M}'_{A}$ and $\mathcal{M}'_{B}$, then  there exists a unique class $\mathcal{W}'_{T}$ such that $\mathcal{M}'_{T}=(\mathfrak{U}_{\mathcal{A}_{A},\widetilde{\mathcal{A}}_{B}},\mathcal{W}'_{T},\mathfrak{T}_{\mathcal{L}_{A},\widetilde{\mathcal{L}}_{B}})$ is a cofibrantly generated abelian model structure on $T$-{\rm Mod} and we have a recollement as shown below
  $$\xymatrixcolsep{4pc}\xymatrix{\mathrm{Ho}(\mathcal{M}'_{A})
  \ar[r]^{\mathrm{L~i_{\ast}~}\cong\mathrm{~R~i_{\ast}}}&
  \ar@<-4ex>[l]_{\mathrm{L~i^{\ast}}}\ar@<3ex>[l]^{\mathrm{R~i^{!}}}\mathrm{Ho}(\mathcal{M}'_{T})
\ar[r]^{\mathrm{L~j^{\ast}~}\cong\mathrm{~R~j^{\ast}}}&\ar@<-4ex>[l]_{\mathrm{L~j_{!}}}\ar@<3ex>[l]^{\mathrm{R~j_{\ast}}}
\mathrm{Ho}(\mathcal{M}'_{B}),}$$
where $\mathrm{L~i^{\ast}}$, $\mathrm{L~i_{\ast}}$, $\mathrm{L~j_{!}}$, $\mathrm{L~j^{\ast}}$, $\mathrm{R~i_{\ast}}$,  $\mathrm{R~j^{\ast}}$, $\mathrm{R~i^{!}}$ and $\mathrm{R~j_{\ast}}$ are the total derived functors of those in (4.1).
\end{lemma}

Let $\mathcal{GP}_{\mathcal{D}_{T}}$ and $\mathcal{GI}_{\mathcal{D}_{T}}$ denote the class of all Gorenstein $\mathcal{D}_{T}$-projective  left $T$-modules and Gorenstein $\mathcal{D}_{T}$-injective right $T$-modules, respectively. Let $R$ be a ring.
The Hovey triple corresponding to the Gorenstein $\mathcal{D}_{R}$-projective model structure is $\mathcal{M}_{R}=(\mathcal{GP}_{\mathcal{D}_{R}},{_{R}\mathcal{W}}, $ $R\text{-}\mathrm{Mod})$. By \cite[Section 4.2]{gj16b}, we know that the homotopy category $\mathrm{Ho}(\mathcal{M}_{R})$ is a triangulated category and it is triangle equivalent to the stable category $\underline{\mathcal{GP}_{\mathcal{D}_{R}}}:=\mathcal{GP}_{\mathcal{D}_{R}}/_{R}\mathcal{P}$, where $_{R}\mathcal{P}$ is the class of projective left $R$-modules. The Hovey triple corresponding to the Gorenstein $\mathcal{D}_{R}$-injective model structure is $\mathcal{M'}_{R}=(\mathrm{Mod}\text{-}R,\mathcal{V}_{R},\mathcal{GI}_{\mathcal{D}_{R}})$. By \cite[Section 4.2]{gj16b}, we know that the homotopy category $\mathrm{Ho}(\mathcal{M}'_{R})$ is triangle equivalent to the stable category $\underline{\mathcal{GI}_{\mathcal{D}_{R}}}:=\mathcal{GP}_{\mathcal{D}_{R}}/\mathcal{I}_{R}$.
\begin{lemma}\label{lem4.3} (1) Let $M$ be a right $B$-module with finite $\mathcal{F}_{2}$-injective dimension and $G$ be a Gorenstein $\mathcal{D}_{B}$-projective left module. Then $\mathrm{Tor}_{i}^{B}(M,G)=0$ for all $i>0$.

(2) Let $M$ be a right $B$-module with finite $\mathcal{F}_{2}$-injective dimension and $E$ be a Gorenstein $\mathcal{D}_{B}$-injective right module. Then $\mathrm{Ext}^{i}_{B}(M,E)=0$ for all $i>0$.
\end{lemma}
\begin{proof}We just prove (1) since (2) follows by a similar way.
  Denote by $\mathrm{fid}(M)$ the $\mathcal{F}_{2}$-injective dimension of $M_{B}$. We shall inducton the $\mathcal{F}_{2}$-injective dimension of $M_{B}$. If fid$(M)=0$ the result is trivial. Let $n>0$ and assume that the result is true for any right $B$-module with $\mathcal{F}_{2}$-injective dimension equal to $n$. Moreover, suppose that fid$(M)=n+1$. Then there exists a short exact sequence of right $B$-modules
$0\rightarrow M \rightarrow I_{0}\rightarrow I_{1}\rightarrow 0$
with $I_{0}$ belong to $\mathcal{F}_{2}$ and $\mathrm{fid}(I_{1})=n$. From this exact sequence, we obtain the following exact sequence:
$$ \mathrm{Tor}^{B}_{i+1}(I_{1},G)\rightarrow \mathrm{Tor}^{B}_{i}(M,G)\stackrel{}\rightarrow \mathrm{Tor}^{B}_{i}(I_{0},G)$$
where $G$ is  Gorenstein $\mathcal{D}_{B}$-projective. By induction hypothesis, we have $\mathrm{Tor}^{B}_{i+1}(I_{1},G)=\mathrm{Tor}^{B}_{i}(I_{0},G)=0$ for all $i > 0$.
Hence $\mathrm{Tor}^{B}_{i}(M,G)=0$ for all $i > 0$ and every Gorenstein $\mathcal{D}_{B}$-projective module $G$ if $M$ has finite $\mathcal{F}_{2}$-injective dimension.
\end{proof}

Let $R$ be any ring. Let $\mathcal{GF}_{\mathcal{D}_{T}}$  denotes the class of all Gorenstein $\mathcal{D}_{T}$-flat  modules. Denote by  $_{R}\mathcal{CT}:={_{R}\mathcal{F}^{\perp}}$ the class of cotorsion left $R$-modules. The Hovey triple corresponding to the Gorenstein $\mathcal{D}_{R}$-flat model structure is $\mathcal{N}_{R}=(\mathcal{GF}_{\mathcal{D}_{R}},{_{R}\mathcal{E}}, {_{R}\mathcal{CT}})$. According to \cite[Proposition 3.1]{eip20}, we see that $\mathcal{GF}_{\mathcal{D}_{R}}\cap\mathcal{GF}_{\mathcal{D}_{R}}^{\perp}={_{R}\mathcal{F}}\cap{_{R}\mathcal{CT}}$.
Its homotopy category $\mathrm{Ho}(\mathcal{N}_{A})$ is  triangle equivalent to the stable category $\underline{\mathcal{GF}_{\mathcal{D}_{R}}\cap {_{R}\mathcal{CT}}}:=(\mathcal{GF}_{\mathcal{D}_{R}}\cap {_{R}\mathcal{CT}}) /(_{R}\mathcal{F}\cap{_{R}\mathcal{CT}})$.
\begin{theorem}\label{the4.3} Assume Setup \ref{set3.3} and suppose that $\Hom_{B}(M,D)$ has finite $\mathcal{C}_{2}$-injective dimension for each $D\in\mathcal{F}_{2}$. Then we have  recollements
$$\xymatrixcolsep{4pc}\xymatrix{\underline{\mathcal{GP}_{\mathcal{D}_{A}}}
  \ar[r]^{\mathrm{L~i_{\ast}~}}&
  \ar@<-4ex>[l]_{\mathrm{L~i^{\ast}}}\ar@<3ex>[l]^{\mathrm{R~i^{!}}}\underline{\mathcal{GP}_{\mathcal{D}_{T}}}
\ar[r]^{\mathrm{L~j^{\ast}~}}&\ar@<-4ex>[l]_{\mathrm{L~j_{!}}}\ar@<3ex>[l]^{\mathrm{R~j_{\ast}}}
\underline{\mathcal{GP}_{\mathcal{D}_{B}}}};\leqno(1)$$
$$\xymatrixcolsep{4pc}\xymatrix{\underline{\mathcal{GI}_{\mathcal{D}_{A}}}
  \ar[r]^{\mathrm{L~i_{\ast}~}}&
  \ar@<-4ex>[l]_{\mathrm{L~i^{\ast}}}\ar@<3ex>[l]^{\mathrm{R~i^{!}}}\underline{\mathcal{GI}_{\mathcal{D}_{T}}}
\ar[r]^{\mathrm{L~j^{\ast}~}}&\ar@<-4ex>[l]_{\mathrm{L~j_{!}}}\ar@<3ex>[l]^{\mathrm{R~j_{\ast}}}
\underline{\mathcal{GI}_{\mathcal{D}_{B}}}};\leqno(2)$$
$$\xymatrixcolsep{4pc}\xymatrix{\underline{\mathcal{GF}_{\mathcal{D}_{A}}\cap {_{A}\mathcal{CT}}}
  \ar[r]^{\mathrm{L~i_{\ast}~}}&
  \ar@<-4ex>[l]_{\mathrm{L~i^{\ast}}}\ar@<3ex>[l]^{\mathrm{R~i^{!}}}\underline{\mathcal{GF}_{\mathcal{D}_{T}}\cap {_{T}\mathcal{CT}}}
\ar[r]^{\mathrm{L~j^{\ast}~}}&\ar@<-4ex>[l]_{\mathrm{L~j_{!}}}\ar@<3ex>[l]^{\mathrm{R~j_{\ast}}}
\underline{\mathcal{GF}_{\mathcal{D}_{B}}\cap {_{B}\mathcal{CT}}}}.\leqno(3)$$

\end{theorem}
\begin{proof}From Lemma \ref{lem4.1}, there are abelian model structures $\mathcal{M}(A)=(\mathcal{GP}_{\mathcal{D}_{A}},{_{A}\mathcal{W}},A\text{-Mod})$ and $\mathcal{M}(B)=(\mathcal{GP}_{\mathcal{D}_{B}},{_{B}\mathcal{W}},B\text{-Mod})$ on $A\text{-Mod}$ and $B\text{-Mod}$, respectively. By Lemma \ref{lem4.3} we have $\mathrm{Tor}_{1}^{B}(M,E)=0$ for any $E\in \mathcal{GP}_{\mathcal{D}_{B}}$. By Theorem \ref{the3.4} and the fact that the abelian module structure $\mathcal{M}(R)=(\mathcal{GP}_{\mathcal{D}_{R}},{_{R}\mathcal{W}}, R\text{-}\mathrm{Mod})$  is projective for any ring, we have
 $$\mathfrak{B}^{\mathcal{GP}_{\mathcal{D}_{A}}}_{\mathcal{GP}_{\mathcal{D}_{B}}}\cap \mathfrak{U}_{_{B}\mathcal{W}}^{_{A}\mathcal{W}}=\mathcal{GP}_{\mathcal{D}_{T}}\cap {_{T}\mathcal{W}}={_{T}\mathcal{P}}=\mathfrak{B}^{_{A}\mathcal{P}}_{_{B}\mathcal{P}}\cap \mathfrak{U}_{B\text{-Mod}}^{A\text{-Mod}}.$$
 It follows that the bimodule $M$ is perfect relative to $\mathcal{M}(A)$ and $\mathcal{M}(B)$. Thus the first recollement  follow from Lemma \ref{the4.6}.
 On the other hand, by Lemma \ref{lem4.1}, there are abelian model structures $\mathcal{M}'_{A}=(\mathrm{Mod}\text{-}A,\mathcal{V}_{A},\mathcal{GI}_{\mathcal{D}_{A}})$ and $\mathcal{M}'_{B}=(\mathrm{Mod}\text{-}B,\mathcal{V}_{B},\mathcal{GI}_{\mathcal{D}_{B}})$ on $\text{Mod-}A$ and $\text{Mod-}B$, respectively. By Lemma \ref{lem4.3} we have  $\mathrm{Ext}^{i}_{A}(M,E)=0$ for any $E\in \mathcal{GI}_{\mathcal{D}_{A}}$. By Theorem \ref{the3.7} and the fact that the abelian module structure $\mathcal{M}'_{R}=(\mathrm{Mod}\text{-}R,\mathcal{V}_{R},\mathcal{GI}_{\mathcal{D}_{R}})$  is injective for any ring, we have
 $$\mathfrak{U}_{\text{Mod-}A,\text{Mod-}B}\cap \mathfrak{T}_{\mathcal{I}_{A},\mathcal{I}_{B}}=\mathcal{I}_{T}=\mathcal{V}_{T}\cap \mathcal{GI}_{\mathcal{D}_{T}} =\mathfrak{U}_{\mathcal{V}_{A},\mathcal{V}_{B}}\cap \mathfrak{T}_{\mathcal{GI}_{\mathcal{D}_{A}},\mathcal{GI}_{\mathcal{D}_{B}}}.$$
 It follows that the bimodule $M$ is coperfect relative to $\mathcal{M}'_{A}$ and $\mathcal{M}'_{B}$. Thus the first recollement  follow from Lemma \ref{lem4.3}.
 Finally, by Lemma \ref{lem4.1}, there are abelian model structures $\mathcal{N}(A)=(\mathcal{GF}_{\mathcal{D}_{A}},{_{A}\mathcal{E}}, \mathcal{CT}_{A})$ and $\mathcal{N}(B)=(\mathcal{GF}_{\mathcal{D}_{B}},{_{B}\mathcal{E}}, \mathcal{CT}_{B})$ on $A\text{-Mod}$ and $B\text{-Mod}$, respectively. By the same proof of Lemma \ref{lem4.3}, we see that $\mathrm{Tor}_{1}^{B}(M,E)=0$ for any $E\in \mathcal{GF}_{\mathcal{D}_{B}}$. By Theorem \ref{the3.8}, we have
 $$\mathfrak{B}^{\mathcal{GF}_{\mathcal{D}_{A}}}_{\mathcal{GF}_{\mathcal{D}_{B}}}\cap \mathfrak{U}_{_{B}\mathcal{E}\cap {_{B}\mathcal{CT}}}^{_{A}\mathcal{E}\cap {_{A}\mathcal{CT}}}=\mathcal{GF}_{\mathcal{D}_{T}}\cap \mathcal{GF}_{\mathcal{D}_{T}}^{\perp} ={_{T}\mathcal{F}}\cap {_{T}\mathcal{CT}}=\mathfrak{B}^{_{A}\mathcal{F}}_{_{B}\mathcal{F}}\cap \mathfrak{U}_{_{B}\mathcal{CT}}^{_{A}\mathcal{CT}}.$$
 It follows that the bimodule $M$ is perfect relative to $\mathcal{N}(A)$ and $\mathcal{N}(B)$. Thus the second recollement  follow from Lemma \ref{the4.6}.
 \end{proof}
 Finally, we give some applications of Theorem \ref{the4.3} for  PGF, Ding injective and Gorenstein flat model structures, respectively.  One can compare it with \cite[Theorem 2.10]{ww21} and \cite[Theorem 4.12]{zhu20}.
\begin{corollary}\label{cor4.4}
 Suppose that $M_{B}$ is finitely presented and  has finite injective dimension, $_{A}M$ has finite flat dimension. Then we have  recollements
$$\xymatrixcolsep{4pc}\xymatrix{\underline{\mathcal{PGF}_{A}}
  \ar[r]^{\mathrm{L~i_{\ast}~}}&
  \ar@<-4ex>[l]_{\mathrm{L~i^{\ast}}}\ar@<3ex>[l]^{\mathrm{R~i^{!}}}\underline{\mathcal{PGF}_{T}}
\ar[r]^{\mathrm{L~j^{\ast}~}}&\ar@<-4ex>[l]_{\mathrm{L~j_{!}}}\ar@<3ex>[l]^{\mathrm{R~j_{\ast}}}
\underline{\mathcal{PGF}_{B}}};\leqno(1)$$
$$\xymatrixcolsep{4pc}\xymatrix{\underline{\mathcal{DI}_{A}}
  \ar[r]^{\mathrm{L~i_{\ast}~}}&
  \ar@<-4ex>[l]_{\mathrm{L~i^{\ast}}}\ar@<3ex>[l]^{\mathrm{R~i^{!}}}\underline{\mathcal{DI}_{T}}
\ar[r]^{\mathrm{L~j^{\ast}~}}&\ar@<-4ex>[l]_{\mathrm{L~j_{!}}}\ar@<3ex>[l]^{\mathrm{R~j_{\ast}}}
\underline{\mathcal{DI}_{B}}};\leqno(2)$$
$$\xymatrixcolsep{4pc}\xymatrix{\underline{\mathcal{GF}_{A}\cap {_{A}\mathcal{CT}}}
  \ar[r]^{\mathrm{L~i_{\ast}~}}&
  \ar@<-4ex>[l]_{\mathrm{L~i^{\ast}}}\ar@<3ex>[l]^{\mathrm{R~i^{!}}}\underline{\mathcal{GF}_{T}\cap {_{T}\mathcal{CT}}}
\ar[r]^{\mathrm{L~j^{\ast}~}}&\ar@<-4ex>[l]_{\mathrm{L~j_{!}}}\ar@<3ex>[l]^{\mathrm{R~j_{\ast}}}
\underline{\mathcal{GF}_{B}\cap {_{B}\mathcal{CT}}}}.\leqno(3)$$

\end{corollary}
\begin{proof} It follows by Theorem \ref{the4.3}, Corollaries \ref{cor3.6}, \ref{cor3.8} and \ref{cor3.11}.
\end{proof}

\bigskip

\renewcommand\refname

\vspace{4mm}
\small
\noindent\textbf{Haiyu Liu}\\
School of Mathematics and Physics, Jiangsu University of Technology, Changzhou 213001, China\\
E-mails: haiyuliumath@126.com\\[1mm]

\noindent\textbf{Rongmin Zhu}\\
School of Mathematical Sciences, Huaqiao University, Quanzhou 362021, China\\
E-mails: rongminzhu@hotmail.com\\[1mm]


{\bf References}
\begin{thebibliography}{99}


\bibitem{br14} D. Bravo, J. Gillespie and M. Hovey, The stable module category of a general
ring, arXiv:1405.5768.


\bibitem{de18} D. Bravo, M. A. P\'{e}rez, Finiteness conditions and cotorsion pairs, J. Pure
Appl. Algebra   221 (6) (2017) 1249-1267.

\bibitem{CS21} M. Cort\'{e}s-Izurdiaga, Jan \v{S}aroch, Module classes induced by complexes and $\lambda$-pure injective modules, preprint, 2021, https://arxiv.org/pdf/2104.08602.pdf
\bibitem{dm08} N. Ding, L. Mao, Gorenstein FP-injective and Gorenstein flat modules, J. Algebra
Appl.  7(4) (2008) 491-506.

\bibitem{ee14} E.  Enochs, M. Cort\'{e}s-Izurdiaga, B. Torrecillas, Gorenstein conditions over triangular matrix rings, J. Pure Appl. Algebra 218 (2014) 1544--1554.

\bibitem{EJ1} E.  Enochs, O. M. G. Jenda, Gorenstein injective and projective modules, Math. Z. 220 (1995) 611-633.
\bibitem{ee00} E. Enochs, O. M. G. Jenda,  Relative Homological Algebra, de Gruyter Exp. Math. 30 Walter de Gruyter Berlin 2000.
\bibitem{fo75}R. M. Fossum, P. Griffith, I. Reiten, Trivial extensions of Abelian categories, in: Homological Algebra of Trivial Extensions of Abelian Categories with Applications to Ring Theory, in: Lect. Notes in Math., vol.456, Springer-Verlag, 1975.
\bibitem{eip20} S. Estrada, A. Iacob, and Marco A. P\'{e}rez, Model structures and relative Gorenstein flat modules and chain complexes, Contemporary Mathematics vol. 751, 2020.

\bibitem{gi10} J. Gillespie, Model structures on modules over Ding-Chen rings, Homology, Homotopy
Appl.  12 (1) (2010) 61-73.
\bibitem{gj16b}J. Gillespie, Hereditary abelian model categories, Bull. Lond. Math. Soc. 48 (2016) 895--922.
\bibitem{gj18}  J. Gillespie, Duality pairs and stable module categories, J. Pure Appl. Algebra 223 (2019) 3425--3435.
\bibitem{gj21} J. Gillespie, A. Iacob, Duality pairs, generalized Gorenstein modules, and Ding injective envelopes,  arXiv: 2105.01770v1.


\bibitem{gr82} E. L. Green, On the representation theory of rings in matrix form, Pacific J. Math. 100 (1) (1982) 138--152.

\bibitem{ah99} A. Haghany, K. Varadarajan, Study of formal triangular matrix rings, Comm. Algebra 27 (11) (1999) 5507--5525.

\bibitem{ah00} A. Haghany, K. Varadarajan, Study of modules over formal triangular matrix rings, J. Pure Appl. Algebra 147 (1) (2000) 41--58.


\bibitem{HJ09}H. Holm, P. J{\o}rgensen, Cotorsion pairs induced by duality pairs. J. Commut. Algebra, 1(4) (2009) 621--633.
\bibitem{hv99} M. Hovey, Model Categories, Mathematical Surveys and Monographs, 63 American Mathematical Society, Providence RI,  1999.
\bibitem{pa17}P. Krylov, A. Tuganbaev, Formal Matrices, Springer International Publishing AG, Gewerbestrasse 11, 6330 Cham, Switzer-land, 2017.

\bibitem{mao20} L.  Mao, Gorenstein flat modules and dimensions over formal triangular matrix rings,  J. Pure Appl. Algebra, 224 (4) (2020) 1-10.

\bibitem{mlx20} L. Mao, Cotorsion pairs and approximation classes over formal triangular matrix rings,
J. Pure Appl. Algebra 224 (6) (2020) 106271, 21 pp.
\bibitem{ma20} L. Mao, Duality pairs and FP-injective modules over formal triangular matrix rings,  Comm. Algebra, 48 (12) (2020)  5296--5310.

\bibitem{mao22} L.  Mao,  Ding modules and dimensions over formal triangular matrix rings,  Rend. Sem. Mat. Univ. Padova 2022.




\bibitem{ne08} A. Neeman, The homotopy category of flat modules, and Grothendieck duality.
    Invent. Math. 174 (2008) 255--308.
\bibitem{Pre09} M. Prest, Purity, spectra and localization, Encyclopedia of Mathematics and its Applications vol. 121, Cambridge University Press, Cambridge, 2009.

\bibitem{cp14} C. Psaroudakis, Homological theory of recollements of abelian categories, J. Algebra 398 (2014) 63--110.

\bibitem{js18} J. \v{S}aroch, J. \v{S}\'{t}ov\'{\i}\v{c}ek, Singular compactness and definability for $\Sigma$-cotorsion and Gorenstein modules, Selecta Math. (N.S.) 26 (2) (2020), Paper No. 23, 40 pp.

\bibitem{ww21} M. Wang, Z. Wang and P. Yang, Recollements from Ding Injective Modules, Bull. Malays. Math. Sci. Soc. (44) (2021) 1459-1469.
\bibitem{wu21}D. Wu, Gorenstein flat-cotorsion modules over formal triangular matrix rings, Bull. Korean Math. Soc. 58 (6) (2021) 1483-1494.
\bibitem{xi12} B.  Xiong, P. Zhang, Gorenstein-projective modules over triangular matrix Artin algebras, J. Algebra Appl. 11 (4) (2012) 1250066, 14 pp.

\bibitem{zh13} P. Zhang, Gorenstein-projective modules and symmetric recollements, J. Algebra 388 (2013) 65--80.

\bibitem{zhu16} R. Zhu, Z. Liu and Z. Wang, Gorenstein homological dimensions of modules over triangular matrix rings, Turk. J. Math. 40 (2016) 146-160.

\bibitem{zhu20} R.  Zhu, Y.  Peng and N.  Ding, {Recollements associated to cotorsion pairs over upper triangular matrix rings}, Publ. Math. Debrecen 98 (1-2) (2021) 83-113.

 \end{thebibliography}
\end{document}